\documentclass{amsart}

\usepackage{amsmath,amssymb,yhmath, url, microtype,color,mathtools}

\makeatletter
\renewcommand\subsection{\@startsection{subsection}{2}%
  \z@{.5\linespacing\@plus\linespacing}{.5\linespacing}%
  {\normalfont\bfseries}}
\makeatother

\newtheorem{lemma}{Lemma}

\newtheorem{cor}[lemma]{Corollary}
\newtheorem{thm}[lemma]{Theorem}
\newtheorem{example}[lemma]{Example}
\newtheorem{thm?}[lemma]{Theorem?}

\newtheorem{remark}[lemma]{Remark}

\newcommand{\ra}{\ensuremath{\rightarrow}}

\newcommand{\Z}{\mathbb{Z}}

\newcommand{\R}{\mathbb{R}}
\newcommand{\PP}{\mathbb{P}}

\newcommand{\C}{\mathbb{C}}
\newcommand{\Q}{\mathbb{Q}}

\newcommand{\OO}{\mathcal{O}}

\newcommand{\Pic}{\operatorname{Pic}}

\begin{document}

\begin{abstract}
Let $D > 546$ be the discriminant of an indefinite rational quaternion algebra.  We show that there are infinitely many 
imaginary quadratic fields $l/\Q$ such that the twist of the Shimura curve $X^D$ by the main Atkin-Lehner involution 
$w_D$ and $l/\Q$ violates the Hasse Principle over $\Q$.
\end{abstract}
\title{Hasse Principle Violations for Atkin-Lehner Twists of Shimura Curves}
\author{Pete L. Clark}
\address{University of Georgia}
\email{plclark@gmail.com}
\author{James Stankewicz}
\address{University of Bristol/Heilbronn Institute for Mathematical Research}
\email{stankewicz@gmail.com}

\maketitle

\section{Introduction}
\noindent
For a number field $k$, we denote by $\mathbb{A}_k$ the adele ring of $k$.
\\ \indent
Let $l/k$ be a quadratic field extension, let $V_{/k}$ be a smooth, projective, geometrically integral variety, and let $\iota_{/k}$ 
be an involution of $V$.  We denote by $\mathcal{T}(V,\iota,l/k)$ the \textbf{quadratic twist} of $V$ by $\iota$ and the quadratic 
extension $l/k$.
We view $V_{/k}$ itself as the ``trivial quadratic twist'' of $V_{/k}$ corresponding to the 
``trivial quadratic extension $k/k$''.  This convention may seem a bit strange at first, but it turns out to be natural: e.g. it makes ``is a quadratic twist'' into an equivalence relation on varieties $V_{/k}$.  We denote by $V/\iota$ the quotient under the action of the group $\{1,\iota\}$.  
\\ \indent
For a curve $X_{/\Q}$ and a $\Q$-rational involution $\iota$, if $\mathcal{P}$ 
is a property of algebraic curves $C_{/\Q}$, then, inspired by terminology of \cite[$\S$1.1]{Mazur-Rubin10} we say \textbf{many quadratic twists of $(X,\iota)$ have property $\mathcal{P}$} if there is $\alpha  \geq 0$ such 
that as $X \ra \infty$ we have 
\[ \# \{ \text{squarefree $d$ with }|d| \leq X \mid \mathcal{T}(X,\iota,\Q(\sqrt{d})/\Q) \text{ has property } \mathcal{P} \} \gg \frac{X}{\log^{\alpha} X}. \]
Let $D > 1$ be a squarefree integer which is a product of an even number of primes.  Let $B/\Q$ be the (unique, up to isomorphism) 
nonsplit indefinite quaternion algebra with reduced discriminant $D$.   Let $(X^D)_{/\Q}$ be the associated Shimura curve, and let $w_D$ be the main Atkin-Lehner involution of $(X^D)_{/\Q}$ (see e.g. \cite[$\S$ 0.3.1]{Clark03}). \\ \indent 
In this note we provide a complement to prior work of the first author \cite{Clark08}, \cite{Clark09}, \cite{ClarkXX} 
using prior work of the second author \cite{Stankewicz14}.  Here is our main result.

\begin{thm}
\label{MAINTHM}
Suppose that the genus of $X^D/w_D$ is at least $2$.  (The set of all such $D$ is given in Lemma \ref{PRELEMMA2} and includes 
all $D > 546$.)  Then many quadratic twists of $(X^D,w_D)$ violate the Hasse Principle.
\end{thm}
\noindent
Let us provide some context.  It is a fundamental problem to find projective varieties $V$ defined over a number field $k$ which violate the Hasse Principle -- i.e., $V(k) = \varnothing$, while for every completion $k_v$ of $k$ we have $V(k_v) \neq \varnothing$ (equivalently, $V(\mathbb{A}_k) \neq \varnothing$). It is also desirable to understand when these violations are explained by the Brauer-Manin obstruction, as is conjectured to hold whenever 
$V$ is a curve \cite[Conjecture 5.1]{Poonen}.  There is a large literature on Hasse Principle violations for curves $V_{/Q}$.  Most examples are sporadic in nature: they apply to one curve at a time. 
\\ \indent
Past work of the first author \cite[Thm. 2]{Clark08} and \cite[Thm. 1]{ClarkXX} gives two versions of a ``Twist Anti-Hasse Principle (TAHP).'' Each gives hypotheses under which for a curve $X$ defined over a number field $k$ and a $k$-rational involution $\iota: X \ra X$, there are infinitely many quadratic field extensions $l/k$ such that $\mathcal{T}(X,\iota,l/k)$ violates the Hasse Principle.   
 As an application, 
it was shown \cite[Thm. 1]{Clark08} that for all squarefree $N > 163$, there are infinitely many primes $p \equiv 1 \pmod{4}$ 
such that $\mathcal{T}(X_0(N),w_N,\Q(\sqrt{p})/\Q))$ violates the Hasse Principle.

\begin{remark}
Bhargava, Gross and Wang show that for each $g \geq 1$, when genus $g$ hyperelliptic curves 
$V_{/Q}$ are ordered by height, a positive proportion violate the Hasse Principle and this violation is explained by Brauer-Manin 
\cite[Thm. 1.1]{BGW13}.  
This is complementary to the Hasse Principle violations obtained 
using TAHP.  One hypothesis of TAHP is the finiteness of $(X/\iota)(k)$, so TAHP does 
not apply to a pair $(X,\iota)$ when $\iota$ is a hyperelliptic involution.  Yet TAHP produces
infinite families of curves violating the Hasse Principle in which the gonality approaches infinity.
\end{remark}
\noindent
It is also natural to pursue Hasse Principle violations on Atkin-Lehner twists of Shimura curves.  However, 
for $D > 1$ the Shimura curve $X^D$ has $X^D(\R) = \varnothing$.  In both previous versions of 
the Twist Anti-Hasse Principle, one of the hypotheses on $X_{/k}$ was $X(\mathbb{A}_k) \neq \varnothing$, so 
the Shimura curves $X^D$ could not be treated directly: instead, \cite[Thm. 3]{Clark08} gives Hasse Principle violations 
on quadratic twists of the Atkin-Lehner quotient $X^D/w_D$ by an Atkin-Lehner involution $w_d$ for $d \neq D$.  \\ \indent
Here we give a further variant of the Twist Anti-Hasse Principle in which the hypothesis 
$X(\mathbb{A}_k) \neq \varnothing$ is weakened to:  
$\mathcal{T}(X,\iota,l/k)(\mathbb{A}_k) \neq \varnothing$ for some quadratic extension $l/\Q$.  This is a quick consequence of \cite[Thm. 1]{ClarkXX}.  To apply our new criterion we need to know that some Atkin-Lehner twist of $X^D_{/\Q}$ has points everywhere locally.  Certainly this is \emph{often} true: whenever the corresponding quaternion algebra $B/\Q$ is split by some 
class number one imaginary quadratic field $l$ with discriminant prime to $D$, then $\mathcal{T}(X^D,w_D,l/\Q)$ has a 
$\Q$-rational CM point \cite[Prop. 65b)]{Clark03}.  That it is true for all $D$ lies deeper and makes use of the results of 
\cite{Stankewicz14}.  
\\ \indent
We give the proof of Theorem \ref{MAINTHM} in $\S$2.  Let \[\mathfrak{D}_D = \{\text{squarefree } d \in \Z \mid 
\mathcal{T}(X^D,w_D,\Q(\sqrt{d})/\Q) \text{ violates the Hasse Principle}\}, \] and for $X \geq 1$, let $\mathfrak{D}_D(X) 
= \# \mathfrak{D}_D \cap [-X,X] = \mathfrak{D}_D \cap [-X,-1]$.  Our proof of Theorem \ref{MAINTHM} shows that $\mathfrak{D}_D(X) \gg \frac{X}{\log X}$.  
In $\S$3 we give a more precise asymptotic analysis of $\mathfrak{D}_D(X)$.  We show first that $\mathfrak{D}_D(X) \gg \frac{X}{\log^{e_D} X}$ for an explicit $e_D < 1$ and then that $\mathfrak{D}_D(X) = O(\frac{X}{\log^{5/8} X})$: in particular
$\mathfrak{D}_D$ has density $0$.  In $\S$4 we discuss relations between our results 
and Hasse Principle violations over quadratic extensions, both in general and with particular attention to the case $(X^D)_{/\Q}$.

\section{The Proof}

\subsection{Another Twist Anti-Hasse Principle}

\begin{thm}(Twist Anti-Hasse Principle, v. III)
\label{TAHPIII}
Let $k$ be a number field.  Let $X_{/k}$ be a smooth, 
projective geometrically integral curve, and let $\iota: X \ra X$ be a $k$-rational automorphism of order $2$.  Let 
$X/\iota$ be the quotient of $X$ under the action of the group $\{e,\iota\}$.  We suppose all of the following: \\
(i) We have $\{ P \in X(k) \mid \iota(P) = P\} = \varnothing$.  \\
(ii) We have $\{P \in X(\overline{k}) \mid \iota(P) = P \} \neq \varnothing \}$.  \\
(iii) We have $\mathcal{T}(X,\iota,l_0/k)(\mathbb{A}_k) \neq \varnothing$ for some quadratic extension $l_0/k$.  \\
(iv) The set $(X/\iota)(k)$ is finite.  \\
Then: \\
a) For all but finitely many quadratic extensions $l/k$, the twisted curve $\mathcal{T}(X,\iota,l/k)_{/k}$ has no $k$-rational points.  \\
b) There are infinitely many quadratic extensions $l/k$ such that the twisted curve $\mathcal{T}(X,\iota,l/k)$ violates 
the Hasse Principle over $k$. \\
c) When $k = \Q$, as $X \ra \infty$, the number of squarefree integers $d$ with $|d| \leq X$ such that  $\mathcal{T}(X,\iota,\Q(\sqrt{d})/\Q)$ violates the Hasse Principle is $\gg \frac{X}{\log X}$.
\end{thm}
\begin{proof}
Let $l/k$ be a quadratic extension, and put 
Let $Y = \mathcal{T}(X,\iota,l/k)$.  Then $\iota$ 
defines a $k$-rational involution on $Y$: indeed, $Y_{/l} \cong X_{/l}$, and if $\sigma$ is the nontrivial field automorphism 
of $l/k$, then $\sigma$ acts on $Y(l) = X(l)$ by $\sigma^* P = \iota(\sigma(P))$, so for all $P \in Y(A)$ and for all $k$-algebras $A$ we have 
\[ \sigma^* \iota (\sigma^*)^{-1} = \iota (\sigma \iota \sigma^{-1}) \iota^{-1} = \iota \iota \iota^{-1} = \iota. \]
The curve $Y/\iota$ is canonically isomorphic to $X/\iota$.  Thus for each quadratic extension $l/k$ -- including the trivial quadratic extension $k/k$ -- there is a map 
\[ \psi_{l}: \mathcal{T}(X,\iota,l/k)(k) \rightarrow (X/\iota)(k). \]
We have 
\begin{equation}
\label{QUADDESCENT1}
 (X/\iota)(k) = \bigcup_{l/k} \psi_{l}( \mathcal{T}(X,\iota,l/k))(l): 
\end{equation}
indeed, if $P \in (X/\iota)(k)$, then $q^*(P) = [Q_1] + [Q_2]$ is an effective $k$-rational divisor of degree $2$.  If $Q_1 = Q_2$, 
then $Q_1$ is a $k$-rational $\iota$-fixed point.  In this case $Q_1$ is also a $k$-rational point of $\mathcal{T}(X,\iota,l/k)$ 
for all quadratic extensions $l/k$.  Otherwise, $Q_1 \neq Q_2 = \iota(Q_1)$, and the Galois action on $\{Q_1,Q_2\}$ determines a unique quadratic extension $l/k$ (possibly the trivial one) such 
that $Q_1,Q_2 \in \mathcal{T}(X,\iota,l/k)$.  This shows (\ref{QUADDESCENT1}) and also shows that under our hypothesis (i) 
the union in (\ref{QUADDESCENT1}) is disjoint.   \\
a) For this part we need only assume hypotheses (i) and (iv).  Let $(X/\iota)(k) = \{P_1,\ldots,P_n\}$.  By hypothesis (i), 
each $P_i$ is either of the form $\psi(Q)$ for $Q \in X(k)$ or $\psi_{\ell_i}(Q)$ for a unique quadratic extension $l_i/k$.  Thus 
there are at most $n$ quadratic extensions $l/k$ such that $\mathcal{T}(X,\iota,l/k)(k) \neq \varnothing$.  \\
b) If we replace (iii) by the hypothesis (iii$'$) $X(\mathbb{A}_k) \neq \varnothing$, then we get (a slightly simplified statement of)
\cite[Thm. 1]{ClarkXX}.   Suppose now that $X(\mathbb{A}_k) = \varnothing$ but for some nontrivial quadratic extension $l_0/k$ 
we have $\mathcal{T}(X,\iota,l_0/k)(\mathbb{A}_k) \neq \varnothing$.  Put $Y = \mathcal{T}(X,\iota,l_0/k)$.    
The canonical bijection $X(\overline{k}) \ra Y(\overline{k})$ induces bijections on the sets of $\iota$-fixed points and of $k$-rational $\iota$-fixed points, so conditions (i) and (ii) hold for $Y$.  By our assumption, hypothesis (iii$'$) holds for $Y$.   And as above we have a canonical isomorphism $(X/\iota) \ra (Y/\iota)$.   So we may apply \cite[Thm. 1]{ClarkXX} to $Y$ in place of $X$, getting the conclusion 
that infinitely many quadratic twists of $(Y,\iota)$ -- equivalently, of $(X,\iota)$ -- violate the Hasse Principle over $k$. \\
c) Similarly, if $k = \Q$ and we replace (iii) by (iii$'$) $X(\mathbb{A}_{\Q}) \neq \varnothing$, then we may apply 
\cite[Thm. 2]{Clark08} to get that the set of prime numbers $p$ such that $\mathcal{T}(X,\iota,\Q(\sqrt{p})/\Q)$ violates 
the Hasse Principle has positive density, and thus that the number of quadratic twists by squarefree $d$ with $|d| \leq X$ 
is $\gg \frac{X}{\log X}$.  This conclusion applies to some quadratic twist $Y = \mathcal{T}(X,\iota,\Q(\sqrt{d_0})/\Q)$ and 
thus it also applies (with a different value of the suppressed constant) to $X$.   
\end{proof}


\subsection{Preliminaries on Shimura Curves}

\begin{thm}
\label{THESISTHMS}
a) (Shimura \cite{Shimura75}) We have $X^D(\R) = \varnothing$.  \\
b) (Clark \cite[Main Theorem 2]{Clark03})
We have $(X^D/w_D)(\mathbb{A}_\Q) \neq \varnothing$.
\end{thm}

\begin{cor}
\label{RCOR}
Let $d \in \Q^{\times} \setminus \Q^{\times 2}$, and let $Y = \mathcal{T}(X^D,w_D,\Q(\sqrt{d})/\Q)$.  Then 
$Y(\R) = \varnothing \iff d > 0$.
\end{cor}
\begin{proof}
If $d > 0$, then $Y \cong_{\R} X^D$, so $Y(\R) = \varnothing$.   By Theorem \ref{THESISTHMS}b), there is $P \in (X^D/w_D)(\R)$.  
Let $q: X^D \ra X^D/w_D$ be the quotient map, defined over $\R$.  Because $X^D(\R) = \varnothing$, the fiber of $q$ consists 
of a pair of $\C$-conjugate $\C$-valued points, say $Q$ and $\overline{Q} = \iota(Q)$.  Thus $Q = \iota(\overline{Q})$, so 
$Q \in Y(\R)$.   
\end{proof}
\noindent
Recall that for $d \in \Z^{< 0}$, there is an order $\OO$ of discriminant $d$ in a quadratic field iff 
$d \equiv 0,1 \pmod{4}$.  If $d \equiv 2,3 \pmod{4}$ we put $h'(d) = 0$.  If $d \equiv 0,1 \pmod{4}$ we define 
$h'(d) = \# \Pic \OO$, the class number of the quadratic order $\OO$ of discriminant $d$.

\begin{lemma}
\label{PRELEMMA1}
Let $D> 1$ be the discriminant of an indefinite quaternion algebra $B_{/\Q}$.   \\
a) The set $\{ P \in X^D(\Q) \mid w_D(P) = P \}$ is empty. \\
b) We have $ \# \{P \in X^D(\overline{\Q}) \mid w_D(P) = P\} = h'(-D) + h'(-4D) > 0$. \\
c) The genus of $X^D$ is 
\begin{equation}
g_D \coloneqq 1 + \frac{\varphi(D)}{12} -  \frac{\prod_{p\mid D} \left(1- \left( \frac{-4}{p} \right) \right)}{4} - 
\frac{ \prod_{p \mid D}\left(1 - \left(\frac{-3}{p} \right) \right)}{3}.
\end{equation}
d) The genus of $X^D/w_D$ is 
\begin{equation}
\label{GENUSEQ}
 1 + \frac{\varphi(D)}{24} - \frac{\prod_{p\mid D} \left(1- \left( \frac{-4}{p} \right) \right)}{8} - 
\frac{ \prod_{p \mid D}\left(1 - \left(\frac{-3}{p} \right) \right)}{6} - \frac{h'(-D) + h'(-4D)}{4}. 
\end{equation}
\end{lemma}
\begin{proof}
a) This is immediate from Theorem \ref{THESISTHMS}a).  b),c),d) See e.g. \cite[$\S$0.3.1]{Clark03}.
\end{proof}

\begin{lemma}
\label{PRELEMMA2}
a) If $D \in \{6,10,14,15,21,22,26,33,34,35,38,39,46,51,55,62,69,74,86,$ \\
$87,94,95, 111,119,134,146,159,194,206\}$, then $X^D/w_D \cong_{\Q} \PP^1$.  \\
b) If $D \in \{57, 58,65,77,82,106,118,122,129,143,166,210,215,314,330,390,510,546\}$, then 
$X^D/w_D$ is an elliptic curve of positive rank. \\
c) For all other $D$ -- in particular, for all $D > 546$ -- the set $(X^D/w_D)(\Q)$ is finite.
\end{lemma}
\begin{proof}
Using (\ref{GENUSEQ}), one sees that $X^D/w_D$ has genus zero iff $D$ is one of the discriminants listed in part a) and 
that $X^D/w_D$ has genus one iff $D$ is one of the discriminants listed in part b).  Thus for all other $D$, $X^D/w_D$ 
has genus at least $2$ and $(X^D/w_D)(\Q)$ is finite by Faltings' Theorem, establishing part c).  By Theorem \ref{THESISTHMS}b) 
the curve $(X^D/w_D)_{/\Q}$ has points everywhere locally, so when it has genus zero it is isomorphic to $\PP^1$, establishing part a).  The case in which $X^D/w_D$ has genus one is handled by work of Rotger; in every case he shows that there is a class number one imaginary quadratic field $K$ and a $\Z_K$-CM point on $X^D$ which induces a 
$\Q$-rational point on $X^D/w_D$, so $X^D/w_D$ is an elliptic curve.
Moreover, Rotger 
identifies $X^D/w_D$ with a modular elliptic curve in Cremona's tables -- see \cite[Table III]{Rotger02}.  All of these 
elliptic curves have rank one.\footnote{This is not a coincidence.  By Atkin-Lehner theory and the Jacquet-Langlands 
correspondence, every elliptic curve which is a $\Q$-isogeny factor of the Jacobian of $X^D$ has odd analytic 
rank.}
This establishes part a).  
\end{proof}

\subsection{Proof of the Main Theorem}
\noindent
Let $D$ be the discriminant of a nonsplit indefinite rational quaternion algebra.  Assume moreover that $D$ does not appear 
in Lemma \ref{PRELEMMA2} a) or b); in particular this holds for all $D > 546$.  We will prove Theorem \ref{MAINTHM} 
by verifying that the pair $(X^D,w_D)$ satisfies the hypotheses of Theorem \ref{TAHPIII}: then by Theorem \ref{TAHPIII}c), 
the number of quadratic twists by $d$ with $d \leq |X|$ which violate the Hasse Principle is $\gg \frac{X}{\log X}$.  
\\ \indent  
Parts a) and b) of Lemma \ref{PRELEMMA1} show that conditions (i) and (ii) hold, and part c) of Lemma \ref{PRELEMMA2} shows that condition (iv) holds.  For $d \in \Q^{\times} \setminus \Q^{\times 2}$, put 
\[ Y_d = \mathcal{T}(X^D,w_D,\Q(\sqrt{d})/\Q). \]
By Corollary \ref{RCOR}, $Y_d(\R) \neq \varnothing \iff d < 0$.  Henceforth we assume that $d < 0$.  
\\ \\
Let $g$ be the genus of $X^D$, and put $\overline{D} = \begin{cases} D & 2 \nmid D \\ \frac{D}{2} & 2 \mid D \end{cases}$.
\\ \\
For $n \in \Z^+$, let $\omega(n)$ be the number of distinct prime divisors of $n$.  Put 
\[e_D = w(\overline{D}) + 2. \]
Let $\mathcal{S}_D$ be the set of prime numbers $\ell$ satisfying: \\ \\
(a) $\ell \equiv 3 \pmod{8}$ and \\
(b) For all primes $q \mid \overline{D}$ we have $\left(\frac{-\ell}{q} \right) = -1$. 
\\ \\
Let $\eta_D$ be the set of all negative integers $d$ such that: \\ \\
(c) $-d = \prod_{i=1}^{2r-1} \ell_i$ is the product of an odd number of distinct primes $\ell_i \in \mathcal{S}_D$ and \\
(d) For all primes $p \in (2,4g^2]$ such that $p \nmid D$, we have $\left(\frac{d}{p} \right) = -1$.  
\\ \\
The set $\eta_D$ is infinite: indeed, by the Chinese Remainder Theorem and Dirichlet's Theorem it contains infinitely many 
elements $d = -\ell$ with $\ell \in \mathcal{S}_D$.  Moreover: \\ \\
(e) For $d \in \eta_D$ we have $d \equiv 5 \pmod{8}$ and thus $2$ is inert in $\Q(\sqrt{d})$.  \\
(f) For $\ell \in \mathcal{S}_D$ we have $\left(\frac{-D}{\ell}\right) = -1$.  To see this, first suppose $2 \nmid D$, so $D = \prod_{i=1}^{2a} q_i$ with $q_1,\ldots,q_{2a}$ distinct odd primes.  Then:
\begin{equation}
\label{LATEREQ}
 \left( \frac{-D}{\ell} \right) = \left( \frac{-1}{\ell} \right) \prod_{i=1}^{2a} \left( \frac{q_i}{\ell} \right) = - \prod_{i=1}^{2a} 
\left( \frac{-\ell}{q_i} \right) = - \prod_{i=1}^{2a} -1 = -1. 
\end{equation}
Now suppose $2 \mid D$, so $D = 2 \prod_{i=1}^{2a-1} q_i$ with $q_1,\ldots,q_{2a-1}$ distinct odd primes.  Then:
\[ \left( \frac{-D}{\ell} \right) = \left(\frac{-1}{\ell} \right) \left(\frac{2}{\ell} \right) \prod_{i=1}^{2a-1} \left( \frac{q_i}{\ell} \right) =
(-1)(-1) \prod_{i=1}^{2a-1} \left( \frac{-\ell}{q_i} \right) = (-1)^{2a-1} = -1. \]
\\
We claim that for all $d \in \eta_D$ we have $Y_d(\mathbb{A}_\Q) = \varnothing$.
Indeed: \\ \\
$\bullet$ As above, since $d < 0$ we have $Y_d(\R) \neq \varnothing$. \\
$\bullet$ By \cite[Thm. 4.1.3]{Stankewicz14} and \cite[Thm. 4.1.5]{Stankewicz14}, we have $Y_d(\Q_p) \neq \varnothing$ 
for all $p \mid d$. \\
$\bullet$ Since $2$ is inert in $\Q(\sqrt{d})$, we have $Y_d(\Q_2) \neq \varnothing$ by either \cite[Thm. 5.1]{Stankewicz14} in the case where $2\mid D$ or \cite[Cor 3.17]{Stankewicz14} when $2\nmid D$.  \\
$\bullet$ If $p \mid \overline{D}$, then $p$ is inert in $\Q(\sqrt{d})$, so by \cite[Cor. 5.2]{Stankewicz14} 
we have $Y_d(\Q_p) \neq \varnothing$.  \\
$\bullet$ If  $p \nmid Dd$ and $p > 4g^2$, then by \cite[Thm. 3.1]{Stankewicz14} we have $Y_d(\Q_p) \neq \varnothing$. \\
$\bullet$ If $p \nmid D$ and $p \in (2,4g^2]$, then $p$ is inert in $\Q(\sqrt{d})$, so by \cite[Cor. 3.17]{Stankewicz14} we 
have $Y_d(\Q_p) \neq \varnothing$. Note that by (d), if $p<4g^2$ then $p\nmid d$.

\section{Analytic Results}
\noindent
Our proof of Theorem \ref{MAINTHM} made use of Theorem \ref{TAHPIII}, and thus as soon as we found a single quadratic twist of $(X^D,w_D)$ with points everywhere locally, 
Theorem \ref{TAHPIII}c) tells us that we have $\gg \frac{X}{\log X}$ quadratic twists up to $X$ violating the Hasse Principle.  But we did not just produce a single such 
twist but rather a family $\eta_D$ of twists which, as we shall now see, has larger size than the $\frac{X}{\log X}$ guaranteed by Theorem \ref{TAHPIII}c).  
\\ \indent
Let $\mathfrak{D}_D = \{\text{squarefree } d \in \Z \mid 
\mathcal{T}(X^D,w_D,\Q(\sqrt{d})/\Q) \text{ violates the Hasse Principle}\}$, and for $X \geq 1$, put 
\[ \mathcal{S}_D(X) = \mathcal{S}_D \cap [1,X], \]
\[ \eta_D(X) = \# \eta_D \cap [-X,-1], \]
\[ \mathfrak{D}_D(X) = \# \mathfrak{D}_D \cap [-X,X] = \# \mathfrak{D}_D \cap [-X,-1]. \]
When $X^D/w_D$ has genus at least $2$ (in particular when $D > 546$), only finitely many twists of $(X^D,w_D)$ have $\Q$-points, and thus $\mathfrak{D}_D(X)$ differs by a constant from the 
function which counts the number of Hasse Principle violations up to $X$.   Since \[-\mathcal{S}_D \subset \eta_D \subset \mathfrak{D}_D, \]
for all $X \geq 1$ we have
\[ \mathcal{S}_D(X) \leq \eta_D(X) \leq \mathfrak{D}_D(X). \]
For $n \in \Z^+$, let $\omega(n)$ be the number of distinct prime factors of $n$.  Put 
\[ e_D = \omega(\overline{D}) + 2. \]
By the Prime Number Theorem in Arithmetic Progressions, as $X \ra \infty$ we have 
\begin{equation}
\label{SDEQ} \# \mathcal{S}_D \cap [1,X] \sim \frac{1}{2^{e_D}} \frac{X}{\log X} 
\end{equation}
and thus 
\[ \liminf_{X \ra \infty} \frac{ \mathfrak{D}_D(X)}{X \log X} \geq 2^{-e_D}, \]
which is aleady a little sharper than what we got by appealing to Theorem \ref{TAHPIII}c).   In this section we will establish the following improved estimates.

\begin{thm} 
\label{ANALYTICTHM}
Fix $D > 1$ an indefinite quaternionic discriminant.  Then: \\
a) Suppose $D \notin \mathcal{E}  = \{6,10,14,15,21,22,33,34,38,46,58,82,94\}$.  Then 
there is a positive constant $c_D$ such that as $X \ra \infty$ we have
\[ \eta_D(X) = c_D \frac{X}{\log^{1-2^{-e_D}}X} + O\left(\frac{X}{\log^{2-2^{-e_D}}X}\right). \]
b) Let $h_D$ be the class number of the field $\Q(\sqrt{-D})$.  As $X \ra \infty$ we have 
\[ \mathfrak{D}_D(X) = O\left( \frac{X}{\log^{1-2^{-e_D} - (2h_D)^{-1}}X} \right) = O\left( \frac{X}{\log^{5/8}(X)} \right). \]
\end{thm}
\noindent
Thus for each $D > 546$, as $X \ra \infty$ there are many more 
quadratic twists of $(X^D,w_D)$ up to $X$ violating the Hasse Principle than 
prime numbers, but the set of such twists has zero asymptotic density.
\\ \\
For an imaginary quadratic discriminant $\Delta$, let $H_{\Delta}(X)$ be the Hilbert class polynomial \cite[p. 285]{Cox89}.  

\subsection{A Preliminary Lemma}

\begin{lemma}
\label{PRELEMMA3}
If $D\not\in \mathcal{E} \coloneqq \{6,10,14,15,21,22,33,34,38,46,58,82,94\}$, then for all prime numbers $p\mid D$, we have $p<4g_D^2$.
\end{lemma}

\begin{proof}  It is an easy consequence of Lemma \ref{PRELEMMA1}c) that \[4g_D^2 \ge \dfrac{1}{36}(12 + \varphi(D) - 7 \cdot 2^{\omega(D)})^2. \] So to get $p < 4 g_D^2$ it will suffice to show that
\begin{equation}
\label{SUFFINEQ}
\forall p \mid D, \ 6 \sqrt p < 12 + \varphi(D) - 7 \cdot 2^{\omega(D)}. 
\end{equation}
\textbf{Case 1}: Suppose $\omega(D) = 2$, so $D = pq$ for prime numbers $p < q$.   For all $q > 66$,  
\[ 12 + \varphi(D) - 7 \cdot 2^{\omega(D)} = (p-1)(q-1) -16 \geq q-17 > 6 \sqrt{q} > 6 \sqrt{p},\]
so (\ref{SUFFINEQ}) holds. Using the genus formula we test all discriminants $D = pq$ with $q \leq 66$; the ones 
for which $q \geq 4 g_D^2$ are precisely the set $\mathcal{E}$.  \\
\textbf{Case 2}: Suppose $\omega(D) \geq 4$.  Put $\beta(D) = \prod_{p \mid D} \frac{p-1}{2}$.  \\ \indent
Let us first suppose that $\beta(D) > 49$.  Then, since $\frac{y-1}{2} \geq \frac{y}{4}$ for all $y \geq 2$, we have 
\[ \beta(D)-7 \geq 6 \sqrt{\beta(D)} \geq 6 \sqrt{\prod_{p \mid D} \frac{p}{4}} \]
and thus for all $p \mid D$ we have 
\[12 + \varphi(D) - 7 \cdot 2^{\omega(D)} > \varphi(D) - 7 \cdot 2^{\omega(D)} \geq 6 \sqrt{D} \geq 6 \sqrt{p}. \]
If $\omega(D) \geq 4$ and $\beta(D) \leq 49$ then $D = p_1 p_2 p_3 p_4$ for primes $p_1 < p_2 < p_3 < p_4 < 101$.  
Using the genus formula we find that $p_4 < 4g_D^2$ in all cases.
\end{proof}
\noindent
Thus for $D \notin \mathcal{E}$, $\eta_D$ is the set of squarefree negative $d \in \Z$ with an odd number of prime divisors all in $\mathcal{S}_D$ such that all primes less than $4g_D^2$ are inert in $\Q(\sqrt d)$. 

\subsection{Proof of Theorem \ref{ANALYTICTHM}a)}
\noindent
Suppose $D \notin \mathcal{E}  = \{6,10,14,15,21,22,33,34,38,46,58,82,94\}$.
\\ \\
Let $I$ be the product of the primes less than 
$4g_D^2$ which do not divide $2D$. Let $\{u_i\} \subset \{1,\ldots,I\}$ be the set of elements such that $\left(\dfrac{u_i}{\ell}\right) = -1$ for all $\ell\mid I$. 
Similarly, let $\{v_j\}$ be the set of elements between 1 and $8\overline D$ such that $v_j \equiv 3\bmod 
8$ and for all $q\mid D$ odd, $\left(\dfrac{v_j}{q}\right) = -1$. The 
number of such $v_j$ is $\frac{\varphi(8\overline D)}{2^{e_D}}$. 
Let $\mu$ be the 
M\"obius function, let $b_n$ be the multiplicative function such that if $p$ 
is prime,

$$b_{p^{m}} = \begin{cases}1 & p\equiv v_j \bmod 8\overline D \textrm{ for some } j  \\ 0 & else.\end{cases}$$

We let $$a_n = b_n\left(\dfrac{1}{\varphi(I)}\sum_{\chi} \sum_i \overline{\chi(-u_i)}\chi(n)\right)\left(\dfrac{\mu^2(n) - \mu(n)}{2}\right),$$ where $\chi$ runs over the $\bmod I$ Dirichlet characters.

Here $a_n = 1$ if and only if $-n \equiv u_i \bmod I$ for some $i$, $n$ (and thus $-n$) is square-free with an odd number of prime factors, and each prime dividing $n$ is congruent to some $v_j$. That is, $a_n$ is the indicator function for $-\eta_D$. 

Consider the function $f(s) = \sum_n a_n n^{-s}$, holomorphic on $\Re(s)>1$. Note that the $a_n$ are not necessarily multiplicative. We however reduce to this case as we write the Dirichlet series $f_{k,\chi}(s) = \sum_{n\ge 0}b_n\mu^k(n)\chi(n) n^{-s}$, again converging in the half-plane $\Re(s)>1$. We therefore have $$f(s) = \dfrac{1}{2\varphi(I)}\sum_\chi\left(\left(\sum_i \overline{\chi(-u_i)}\right)(f_{2,\chi}(s) -f_{1,\chi}(s))\right).$$ We begin by showing that with the exception of $(k,\chi) =(2,\mathbf{1})$, these are in fact holomorphic in the region $\Re(s)\ge 1$.

Consider \begin{eqnarray*}
\log(f_{k,\chi}(s)) & = & \sum_{p} \log(\sum_{m\ge 0} b_{p^m}\mu^k(p^m)\chi(p^m)p^{-ms})
\\ &= & \sum_{p}\log(1 + b_p(-1)^k\chi(p)p^{-s})
\\ & = & (-1)^k\sum_p \frac{b_p\chi(p)}{p^s} + \beta_{k,\chi}(s)
\end{eqnarray*}

where $\beta_{k,\chi}(s)$ is holomorphic on $\Re(s)>1/2$.

Now use the fact that $$b_p = \frac{1}{\varphi(8\overline D)} \sum_{\psi \bmod 8\overline D}\sum_j \overline{\psi(v_j)} \psi(p).$$

Therefore $$\log(f_{k,\chi}(s)) = (-1)^k\frac{1}{\varphi(8\overline D)}\sum_{\psi}\sum_j\overline{\psi(v_j)} \log(L(s,\chi\psi)) + \rho_{k,\chi}(s),$$ where $\rho_{k,\chi}$ is holomorphic for $\Re(s)>1/2$.

It follows that zero-free regions for $L$-functions of Dirichlet characters give zero-free regions for the $f_{j,\chi}$ and thus holomorphic regions for $f$. In particular, if $\epsilon$ is a Dirichlet character and $\delta_\epsilon =1$ for $\epsilon = \mathbf{1}$ (the trivial 
character) and zero otherwise then there are positive numbers $A_\epsilon, B_\epsilon$ such that $\log(L(s,\epsilon)) - \delta_\epsilon\log(1/(s-1))$ is holomorphic on $\Re(s)\ge 1 - B_\epsilon/\log^{A_\epsilon}(2 + |\Im(s)|)$ \cite[Proposition 1.7]{Serre76}.

Now we note that since $(I,8\overline D)=1$, $\chi\psi = \mathbf{1}$ if and only if $\chi= \mathbf{1}$ and $\psi = \mathbf{1}$. Therefore by exponentiating, we find a holomorphic, nonzero function $g_{k,\chi}$ on the same region in $\C$ such that \[f_{k,\chi}(s) = \left(\dfrac{1}{s-1}\right)^{\delta_\chi (-1)^k/2^{e_D}}g_{k,\chi}(s).\]

Thus there is a function $g$ holomorphic on the intersection of the 
$A_\epsilon, B_\epsilon$ regions such that $f(s) = 
\left(\dfrac{1}{s-1}\right)^{2^{-e_D}}g(s).$ Finally, we may 
apply the method of Serre and Watson \cite[Th\'eor\`eme 2.8]{Serre76} to get our asymptotic 
for $\sum_{n\le X}a_n = \#\eta_D(X)$.


\subsection{Proof of Theorem \ref{ANALYTICTHM}b)}
\noindent
We define a set $\mathcal{S}'_D$ of primes, as follows:
\\ \\
$\bullet$ If $D \not \equiv 3 \pmod{4}$, then $\ell \in \mathcal{S}'_D$ iff $\left(\frac{-D}{\ell} \right) = 1$ 
and $H_{-4D}(X)$ has a root modulo $\ell$.  \\
$\bullet$ If $D \equiv 3 \pmod{4}$, the $\ell \in \mathcal{S}'_D$ iff $\left(\frac{-D}{\ell} \right) = 1$ and 
at least one of $H_{-D}(X)$ and $H_{-4D}(X)$ has a root modulo $\ell$.
\\ \\
By (\ref{LATEREQ}) the sets $\mathcal{S}_D$ and $\mathcal{S}'_D$ are disjoint.  Moreover, by 
\cite[Thm. 4.1]{Stankewicz14}, if $d \in \mathcal{D}_D$, then for all primes $p \mid d$ we have 
\[ p \in \mathcal{C}_D \coloneqq \mathcal{S}_D \cup \mathcal{S}'_D \cup \{\text{prime divisors of } 2D\}. \]
Step 1: We show that $\mathcal{S}'_D$ is a Chebotarev set of density $\frac{1}{2h_D}$. \\
For any imaginary quadratic discriminant $\Delta < 0$, the field $K_{\Delta} = \Q(\sqrt{-\Delta})[X]/(H_{\Delta}(X))$ is the ring class field 
of discriminant $\Delta$.  This field is Galois over $\Q$ of degree twice the class number of the imaginary 
quadratic order of disdcriminant $\Delta$.  \\
$\bullet$ Suppose $D \not \equiv 3 \pmod{4}$.  Then up to a finite set, $\mathcal{S}'_D$ is the set of 
primes which split completely in $K_{-4D}$, which in this case is the Hilbert class 
field of $\Q(\sqrt{-D})$.  Thus $\mathcal{S}'_D$ is Chebotarev of density $\frac{1}{2h_D}$.  \\
$\bullet$ Suppose $D \equiv 3 \pmod{4}$.  Then $K_{-D}$ is the Hilbert 
class field of $\Q(\sqrt{-D})$.  The ring class field $K_{-4D}$ contains $K_{-D}$, so the set of primes splitting 
completely in $K_{-D}$ or in $K_{-4D}$ is the same as the set of primes splitting completely in $K_{-D}$.  
Thus again $\mathcal{S}'_D$ is Chebotarev of density $\frac{1}{2h_D}$. \\ 
Step 2: Since $\mathcal{S}_D$ is a Chebotarev set of density $\frac{1}{2^{e_D}}$ that is disjoint from $\mathcal{S}'_D$ 
and $\{\text{prime divisors of } 2D\}$ is finite, it follows that $\mathcal{C}_D$ is a Chebotarev set of density 
\[\delta_D = \frac{1}{2^{e_D}} + \frac{1}{2h_D}. \]  Let $\mathfrak{D}'_D$ 
be the set of $n \in \Z^+$ with all prime divisors lying in $\mathcal{C}_D$, and for $X \geq 1$, put 
$\mathfrak{D}'_D(X) = \mathfrak{D}'_D \cap [1,X]$.  By \cite[Thm. 2.8]{Serre76} we have 
that as $X \ra \infty$, 
\[ \mathfrak{D}'_D(X) = O\left(\frac{X}{\log^{1-\delta_D}X}\right). \]
Since $-\mathfrak{D}_D \subset \mathfrak{D}'_D$, we get 
\[ \mathfrak{D}_D(X) = O\left(\frac{X}{\log^{1-\delta_D}X}\right). \]
Step 3: By the genus theory of binary quadratic forms \cite[Prop. 3.11]{Cox89}, $h_D$ is even.  Since $e_D = \omega(\overline{D}) +2 \geq 3$, 
we get $\delta_D \leq \frac{3}{8}$, so $\mathfrak{D}_D(X) = O\left( \frac{X}{\log^{5/8}X} \right)$.

\section{Final Remarks}
\noindent
Most results on Shimura curves violating the Hasse Principle concentrate on the case of finding number fields $k$ 
such that $(X^D)_{/k}$ violates the Hasse Principle.  Since $X^D(\R) = \varnothing$, such a $k$ cannot have a real place, 
and thus the case of an imaginary quadratic field is in a certain sense minimal.   Here is such a result.

\begin{thm} 
(Clark \cite[Thm. 1]{Clark09}) 
\label{ISRAELTHM1}
If $D > 546$, then there are infinitely many quadratic fields $l/\Q$ such that $(X^D)_{/l}$ violates the Hasse Principle.
\end{thm}
\noindent
The global input of Theorem \ref{ISRAELTHM1} is a result of Harris-Silverman \cite[Cor. 3]{Harris-Silverman91}: if $X_{/k}$ is a curve, then $X$ has infinitely many quadratic points iff $X$ admits a degree $2$ $k$-morphism to $\PP^1$ or to an elliptic curve of positive 
rank.\footnote{Harris-Silverman state their result under the hypothesis that $X$ admits no degree $2$ morphism to $\PP^1$ or to 
a curve of genus one.  Their argument immediately gives the stronger result, as has been noted by several authors, e.g. 
\cite[Thm. 8]{Rotger02}.  Note that the result relies on an extraordinarily deep theorem of Faltings classifying $k$-rational points on subvarieties of abelian varieties \cite{Faltings94}.}  In \cite{Rotger02}, Rotger shows that the Shimura curves 
$X^D$ with infinitely many quadratic points are precisely those in which $X^D/w_D$ is either $\PP^1$ or an elliptic curve, i.e., 
the values of $D$ recorded in parts a) and b) of Lemma \ref{PRELEMMA2}.  Thus under the hypotheses of Theorem \ref{MAINTHM}, 
for all but finitely many quadratic fields $l$ such that $\mathcal{T}(X^D,w_D,l/\Q)$ violate the Hasse Principle over $\Q$, also 
$(X^D)_{/l}$ violates the Hasse Principle: we recover Theorem \ref{ISRAELTHM1}.  
\\ \indent
Assuming the hypotheses of Theorem \ref{TAHPIII} and also that $X(k) = \varnothing$ and $X$ admits no degree $2$ morphism to $\PP^1$ or to an elliptic curve $E_{/k}$ with positive rank, we deduce that there are infinitely many quadratic extensions $l/k$ such that $X_{/l}$ violates the Hasse Principle.  But this was known already.  Indeed:

\begin{thm}
\label{ISRAELTHM7}
(Clark \cite[Thm. 7]{Clark09})
Let $X_{/k}$ be a smooth, projective geometrically integral curve over a number field.  Assume: \\
(i) We have $X(k) = \varnothing$.  \\
(ii) There is no degree $2$ morphism from $X$ to $\PP^1$ or to an elliptic curve $E_{/k}$ with positive rank.  \\
(iii) For every place $v$, the curve $X_{/K_v}$ over the completion $K_v$ of $K$ at $v$ has a closed point of degree at most $2$. \\
Then there are infinitely many quadratic extensions $l/k$ such that $X_{/l}$ violates the Hasse Principle.
\end{thm}
\noindent
As in \cite{Clark09}, for a curve $X$ over a field $k$, we denote by $m(X)$ the least degree of a closed point on $X$.  For a curve $X$ defined over a 
number field $k$, we put  
\[ m_{\operatorname{loc}}(X) = \sup_v m(X_{/K_v}), \]
the supremum extending over all places of $K$.  Then condition (iii) in Theorem \ref{ISRAELTHM7} can be restated as 
$m_{\operatorname{loc}}(X) \leq 2$.  For a curve equipped with an involution defined over a number field 
$(X,\iota)_{/k}$, consider the following local hypotheses:
\\ \\
(L1) $\mathcal{T}(X,\iota,l/k)(\mathbb{A}_k) \neq \varnothing$ for some quadratic extension $l/k$ (we allow $l = k$).  \\
(L2) $(X/\iota)(\mathbb{A}_k) \neq \varnothing$.  \\
(L3) $m_{\operatorname{loc}}(X) \leq 2$.
\\ \\
Clearly (L1) $\implies$ (L2) $\implies$ (L3).  Whereas Theorem \ref{ISRAELTHM7} uses the weakest condition (L3)
to get Hasse Principle violations over quadratic extensions $l/k$, Theorem \ref{MAINTHM} uses 
the strongest condition (L1) to get Hasse Principle violations over $k$.  
\\ \indent
To prove Theorem \ref{ISRAELTHM1} it suffices to establish that 
$m_{\operatorname{loc}}(X^D) \leq 2$ (and thus clearly $m_{\operatorname{loc}}(X^D) = 2$ since $X^D(\R) = \varnothing$).  
This is \cite[Thm. 8a]{Clark09}.  We want to emphasize that the local analysis from \cite{Stankewicz14} needed to show 
that $X^D$ satisfies (L1) lies considerably deeper.  Thus it is our perspective that Theorem \ref{MAINTHM}, which gives 
Hasse Principle violations over $\Q$, is a deeper result than Theorem \ref{ISRAELTHM1}, which gives Hasse Principle violations 
over quadratic fields.
\\ \\
It is natural to ask whether (L1), (L2) and (L3) may in fact be equivalent.  

\begin{example}
Lemma \ref{PRELEMMA2}a) gives $X^{55}/w_{55} \cong \PP^1$.  Now consider the Atkin-Lehner 
involution $w_5$ on $X^{55}$.  A result of Ogg implies that since $\left(\frac{11}{5}\right) = 1$, we have 
$(X^{55}/w_5)(\R) = \varnothing$.  (See \cite[Thm. 57 and Cor. 42]{Clark03}.)   Since $(X^{55},w_{55})_{/\Q}$ 
satisfies (L2), the curve $X^{55}/\Q$ satisfies (L3).  So $(X^{55},w_5)_{/\Q}$ satisfies (L3) but not (L2).
\end{example}
\noindent
Notice that $(X/\iota)(k) \neq \varnothing$ 
implies (L1), so a counterexample to (L2) $\implies$ (L1) would yield a curve $(X/\iota)_{/k}$ which violates the Hasse Principle.   Whether such counterexamples exist we leave as an open question. \\ \indent These kinds of Hasse Principle violations seem to lie deeper still than Hasse Principle violations on quadratic twists of 
$(X,\iota)$ over $k$.  In the case of Shimura curves $X^D$ we know that $(X^D/w_D)(\mathbb{A}_k) \neq \varnothing$, and suitable CM points on $X^D$ induce $\Q$-points on $X^D/w_D$.  But we may choose $D$ so that no class 
number one imaginary quadratic field splits the associated quaternion algebra $B/\Q$, and then the folk wisdom is that 
we ought to have $(X^D/w_D)(\Q) = \varnothing$ for ``most'' such $D$ (perhaps all but finitely many).  This is a close analogue of the problem of determining all $\Q$-points on $X_0(N)/w_N$, and both are wide open. 
 It seems striking that we can establish Hasse Principle violations on Shimura curves using a kind of descent via the map $X^D \ra X^D/w_D$ without knowing whether $X^D/w_D$ has any $\Q$-points.

\end{document}